\documentclass[a4paper,11pt,reqno]{amsart}
\usepackage[left=1in,right=1in,top=1in,bottom=1in]{geometry}
\usepackage{enumitem}
\usepackage{amssymb, amsmath,latexsym,amsfonts,amsbsy, amsthm,mathtools,graphicx,color}
\usepackage{float}
\usepackage{tikz}
\usetikzlibrary{arrows, decorations.markings,matrix,trees}
\usepackage{hyperref}
\hypersetup{
   colorlinks=true,
   citecolor=blue,
   filecolor=blue,
   linkcolor=blue,
   urlcolor=blue
}

\def\XXint#1#2#3{{\setbox0=\hbox{$#1{#2#3}{\int}$}
  \vcenter{\hbox{$#2#3$}}\kern-.5\wd0}}

\newcommand{\lda}{\lambda}
\newcommand{\om}{\Omega}            
\newcommand{\pa}{\partial}
\newcommand{\va}{\varepsilon}       
\newcommand{\ud}{\mathrm{d}}
\newcommand{\be}{\begin{equation}} 
\newcommand{\ee}{\end{equation}}
\newcommand{\w}{\omega}      
\newcommand{\Lda}{\Lambda}    

\newcommand{\A}{\mathbf{A}}
\newcommand{\cA}{\mathcal{A}}

\newcommand{\B}{\mathbf{B}}

\newcommand{\CC}{\mathbf{C}}

\newcommand{\cD}{\mathcal{D}}

\newcommand{\I}{\mathbf{I}}

\newcommand{\cL}{\mathcal{L}} 

\newcommand{\Z}{\mathbb{Z}}

\newcommand{\MM}{\mathbb{M}}

\newcommand{\m}{\mathbf{m}}

\newcommand{\cN}{\mathcal{N}}

\newcommand{\n}{\mathbf{n}}

\newcommand{\PP}{\mathbf{P}}

\newcommand{\Q}{\mathbf{Q}}  

\newcommand{\R}{\mathbb{R}}

\newcommand{\Ss}{\mathbb{S}}

\newcommand{\uu}{\mathbf{u}}

\newcommand{\T}{\mathrm{T}}

\newcommand{\sg}{\sigma} 
\newcommand{\ift}{\infty} 
\newcommand{\wt}{\widetilde}

\newcommand{\f}{\frac}

\newcommand{\ol}{\overline}

\newcommand{\op}{\operatorname}

\newcommand{\na}{\nabla}

\DeclareMathOperator{\dist}{dist}

\DeclareMathOperator{\Bad}{Bad}
\DeclareMathOperator{\diam}{diam}

\DeclareMathOperator{\supp}{supp}

\DeclareMathOperator{\tr}{tr}

\DeclareMathOperator{\sing}{sing}

\DeclareMathOperator{\loc}{loc}

\def\<{\langle}\def\>{\rangle}
\def\({\left(}\def\){\right)}
\def\[{\left[}\def\]{\right]}
\numberwithin{equation}{section}
\theoremstyle{plain}
\newtheorem{thm}{Theorem}[section]
\newtheorem{cor}[thm]{Corollary}

\newtheorem{lem}[thm]{Lemma}
\newtheorem{prop}[thm]{Proposition}

\theoremstyle{definition}
\newtheorem{defn}[thm]{Definition}

\theoremstyle{remark}
\newtheorem{rem}[thm]{Remark}

\title[Improved convergence of Landau-de Gennes minimizers]{Improved convergence of Landau-de Gennes minimizers in the vanishing elasticity limit}
\author{Haotong Fu}
\address{School of Mathematical Sciences, Peking University, Beijing 100871, China}
\email{547434974@qq.com}

\author{Huaijie Wang}
\address{School of Mathematical Sciences, Peking University, Beijing 100871, China}
\email{huaijie\_wang@163.com}

\author{Wei Wang}
\address{School of Mathematical Sciences, Peking University, Beijing 100871, China}
\email{gjmtamag@gmail.com,\,\,2201110024@stu.pku.edu.cn}

\date{}

\begin{document}

\begin{abstract}
We investigate the vanishing elasticity limit for minimizers of the Landau-de Gennes model with finite energy. By adopting a refined blow-up and covering analysis, we establish the optimal $ L^p $ ($ 1<p<+\ift $) convergence of minimizers and achieve the sharp $ L^1 $ convergence rate of the bulk energy term.
\end{abstract}
\maketitle

\section{Introduction}
\subsection{Landau-de Gennes model}	
Nematic liquid crystals are characterized by local orientational order among their constituent molecules. While the molecules move freely, as in an isotropic fluid, they tend to align along locally preferred directions. Several continuum theories have been proposed to model this orientational behavior, each characterized by a distinct choice of order parameter (see \cite{DG93, Eri91, FS08, Fra58}). Among these, the Landau–de Gennes theory \cite{DG93} is the most comprehensive and widely accepted framework for describing the behavior of nematic liquid crystals. The theory employs the so-called $\Q$-tensor---a $3\times 3$ symmetric, traceless matrix---as the order parameter, with the analysis conducted in the five-dimensional space $\mathbb S_0$ of such tensors:
$$
\Ss_0:=\{\Q\in\mathbb{M}^{3\times 3}:\Q^{\T}=\Q,\,\,\tr\Q=0\}.
$$
The central object in the Landau-de Gennes theory is the free energy functional $F(\Q, \om)$. Stable equilibrium configurations of the liquid crystalline system with a domain $\om\subset\mathbb R^n$ ($n=2, 3$) correspond to local minimizers of this functional. The simplest form of the Landau-de Gennes energy is given by
$$
F(\Q,\om):= \int_\om(f_e(\Q)+f_b(\Q))\ud x.
$$
The elastic energy density $ f_e(\Q)$ characterizes the inhomogeneity of the alignment of the liquid crystal molecules. For elastic constants $L_i$ ($i=1,2,3$) depending on the material, we express $ f_e(\Q) $ as
$$
f_e(\Q)=\f{L_1}{2}|\nabla\Q|^2+\f{L_2}{2}\pa_j\Q_{ij}\pa_k\Q_{ik}+\f{L_3}{2}\pa_k\Q_{ij}\pa_j\Q_{ik}.
$$
The simplest form of the bulk potential density $f_b(\Q)$ takes the following quartic polynomial
$$
f_b(\Q)=-\f{a}{2}\tr{\Q}^2-\f{b}{3}\tr\Q^3+\f{c}{4}(\tr\Q^2)^2,
$$
where $a\ge 0, b,c >0$ are material constants.

Let $\Q\in\mathbb S_0$ and $\lambda_1,\lambda_2,\lambda_3$ be the eigenvalues of $\Q$. Then, we have
$$
\Q=\lambda_1\mathbf n_1\otimes\mathbf n_1+\lambda_2\mathbf n_2\otimes\mathbf n_2+\lambda_3\mathbf n_3\otimes\mathbf n_3,
$$
where $\mathbf n_1,\mathbf n_2,\mathbf n_3\in \mathbb S^2$ with $\mathbf n_i\cdot\mathbf n_j=0\,(1\le i,j\le 3)$. Since $Q_{ii}=0$, we can rewrite $\Q$ as
$$
\Q=s\left(\mathbf n\otimes\mathbf n-\f{1}{3}\mathbf I\right)+r\left(\mathbf m\otimes\mathbf m-\f{1}{3}\mathbf I\right),
$$
where $s,r\in \mathbb R$ and $\mathbf n,\mathbf m$ are eigenvectors of $\Q$ with $\mathbf n\cdot\mathbf m=0$. One can classify $\Q$ into three phases: When $s=r=0$, i.e., $\Q=0$, then $\Q$ is said to be isotropic. When $s$ and $r$ are different and non-zero, it is said to be biaxial. When $s=r\neq0$ or $s=0,r\neq0$ or $s\neq0,r=0$, it is said to be uniaxial and $\Q$ can be rewritten as
$$
\Q=t\left(\uu\otimes\uu-\f{1}{3}\mathbf I\right), \quad \uu\in\mathbb S^2.
$$

\subsection{Vanishing elasticity limit}
Let $L_2=L_3=0$ for simplicity. Since the elastic constants are so small compared to the bulk constants, we consider the rescaled energy functional with a small parameter $\va>0$. Let $ \om\subset\R^3 $ be a bounded domain. We consider the Landau-de Gennes energy 
\be
E_{\va}(\Q,\om):=\int_{\om}e_{\va}(\Q)\ud x,\quad e_{\va}(\Q):=\f{1}{2}|\na\Q|^2+\f{1}{\va^2}f(\Q),\label{LdG}\tag{LdG}
\ee
where
$$
f(\Q)=k-\f{a}{2}\tr\Q^2-\f{b}{3}\tr\Q^3+\f{c}{4}(\tr \Q^2)^2,\quad\Q\in\Ss_0.
$$
Here, $ k $ is an additive constant such that $ \inf_{\Q\in\Ss_0}f(\Q)=0 $. 

The vacuum manifold is
$$
\mathcal{N}:=\left\{s_*\(\n\otimes\n-\f{1}{3}\I\):\n\in\mathbb{S}^2\right\}=f^{-1}(0),
$$
where
$$
s_*:=s_*(a,b,c)=\f{b+\sqrt{b^2+24ac}}{4c}.
$$
Indeed, let $\va\to 0^+$, the term $\f{1}{\va^2}f(\Q)$ in \eqref{LdG} forces the minimizers to take the value in the vacuum manifold. The limiting energy functional is given by
\be
E(\Q,\om):=\int_{\om}|\na\Q|^2\ud x,\quad\Q\in H^1(\om,\cN).\label{Dirichletenergy}\tag{Dir}
\ee

In \cite{MZ10}, Majumdar and Zarnescu investigated the asymptotic behavior of minimizers of the Landau–de Gennes energy \eqref{LdG} as the elastic parameter $\varepsilon>0$ tends to zero. They showed that the minimizer $\Q_\va$ converges in $H^1$ to a minimizer of the limiting Dirichlet energy \eqref{Dirichletenergy}. Moreover,  outside a finite set of point singularities of the limiting map, the convergence holds uniformly on compact sets. In a subsequent work \cite{NZ13}, Nguyen and Zarnescu improved the uniform convergence result to $C^j$ convergence with $j\in \mathbb Z_+$. More recently, Geng and Zarnescu \cite{GZ23} established a refined convergence result, demonstrating that the minimizer of \eqref{LdG} can be well approximated outside an $O(\varepsilon)$-neighborhood of the point defect. The results in \cite{MZ10, NZ13} were obtained under the assumption of uniformly bounded energy, namely,
$$
E_\varepsilon\left(\Q_\varepsilon,\om\right)\le C
$$
for some $C>0$ independent of $\varepsilon$. This assumption was relaxed in \cite{Can17}, where Canevari extended the analysis to the case
$$
E_\varepsilon(\Q_\varepsilon,\om)\leq C\left(|\log\va|+1\right),
$$
allowing for more complex singular sets comprising both point and line defects, see \cite{WZ24} for the problem with the biaxial setting. Researchers have extensively studied the vanishing elasticity limit for \eqref{LdG}; see, for instance, \cite{DMP21} for torus-like defect structure, \cite{CL22, FH22} for models with non-zero $L_2$ and $L_3$, and \cite{HL22, WWZ17} for analyses in the gradient flow setting.

\subsection{Main results}
We further enhance the convergence results of \cite{MZ10,NZ13} by establishing an $L^{3,\infty}$ estimate for the gradient, which in turn guarantees $L^p$-convergence of the minimizers. The $C^j$-convergence in \cite[Theorem 2]{NZ13} relies on the smoothness of the limiting map, whereas our result accommodates singularities in the limit. We begin by introducing the definition of local minimizers.
\begin{defn}[Local minimizer]
Let $ \va\in(0,1) $. $ \Q_{\va}\in H_{\loc}^1(\om,\Ss_0) $ is a local minimizer of \eqref{LdG} if for any $ B_r(x)\subset\subset\om $ and $ \PP\in H^1(B_r(x),\Ss_0) $ with $ \PP=\Q_{\va} $ on $ \pa B_r(x) $ in the sense of trace,
$$
E_{\va}(\Q_{\va},B_r(x))\leq E_{\va}(\PP,B_r(x)).
$$
With parallel statements, we also define the local minimizer of \eqref{Dirichletenergy}.
\end{defn}

Our main theorem is as follows.

\begin{thm}\label{main}
Let $ M>0 $. Suppose that $ \{\Q_{\va}\}_{\va\in(0,1)}\subset H_{\loc}^1(\om,\Ss_0) $ are local minimizers of \eqref{LdG} such that
\be
\sup_{\va\in(0,1)}\(E_{\va}(\Q_{\va},\om)+\|\Q_{\va}\|_{L^{\ift}(\om)}\)\leq M.\label{conditionboundary}
\ee
Then, there is $ \Q_0\in H^1(\om,\cN) $ a local minimizer of \eqref{Dirichletenergy} such that for some $ \va_i\to 0^+ $, $ \Q_{\va_i}\to\Q_0 $ strongly in $ H_{\loc}^1(\om,\Ss_0) $, and the following properties hold.
\begin{enumerate}[label=$(\theenumi)$]
\item For any $ p\in(1,+\ift) $ and $ K\subset\subset\om $,
\be
\lim_{i\to+\ift}\|\Q_{\va_i}-\Q_0\|_{L^p(K)}=0.\label{Lpconvergence}
\ee
\item For any $ K\subset\subset\om $,
\be
\int_{K}\f{1}{\va_i^2}f(\Q_{\va_i})\ud x\leq C\va_i,\label{fQconvergence}
\ee
where $ C>0 $ depends only on $ a,b,c,K $, and $ M $. 
\end{enumerate}
\end{thm}

Several remarks are in order.

\begin{rem}
The convergence rate \eqref{fQconvergence} is not only correct for a subsequence, but indeed, we prove that it holds for all $\Q_{\va}$. The convergence results in \eqref{Lpconvergence} and \eqref{fQconvergence} are already optimal. We refer to Section \ref{sharpsection} for more details on their sharpness.
\end{rem}

\begin{rem}
We believe that our method applies to other settings, such as generalized Ginzburg-Landau \cite{MRS21} and torus-like solutions of the Landau-de Gennes model \cite{DMP21}.
\end{rem}

\begin{rem}
The condition \eqref{conditionboundary} is satisfied if $ \om $ is regular enough and the $ \Q_{\va} $ is a global minimizer with nice boundary conditions. One can refer to \cite{MZ10} for example. Precisely, as a direct consequence, we have the following result under the setting of \cite{MZ10}.
\end{rem}

\begin{cor}\label{Maincor}
Let $ \om\subset\R^3 $ be a smooth bounded domain and $ \Q_b\in C^{\ift}(\pa\om,\cN) $. Assume that $ \{\Q_{\va}\}_{\va\in(0,1)} $ is a sequence of $($global$)$ minimizers of the boundary value problem
\be
\min\{E_{\va}(\Q,\om):\Q_{\va}\in H^1(\om,\Ss_0),\,\,\Q=\Q_b\text{ on }\pa\om\text{ in the sense of traces}\}.\label{globalminimizer}
\ee
Then there exists $ \Q_0\in H^1(\om,\cN) $, a $($global$)$ minimizer of
\be
\min\{E(\Q,\om):\Q\in H^1(\om,\cN),\,\,\Q=\Q_b\text{ on }\pa\om\text{ in the sense of traces}\},\label{globalminimizer2}
\ee
and a subsequence $ \va_i\to 0^+ $ such that $ \Q_{\va_i}\to\Q_0 $ strongly in $ H^1(\om,\Ss_0) $. Moreover,
\be
\lim_{i\to+\ift}\|\Q_{\va}-\Q_0\|_{L^p(\om)}=0\label{cor1}
\ee
for any $ p\in(1,+\ift) $, and
\be
\int_{\om}\f{1}{\va_i^2}f(\Q_{\va_i})\ud x\leq C\va_i.\label{cor2}
\ee
where $ C>0 $ depends only on $ a,b,c,\Q_b $, and $ \om $.
\end{cor}

\subsection{Difficulties and strategies}
Previous works (e.g.,\cite{FH22,MZ10,NZ13}) established refined convergence results under the assumption that the limiting configuration is regular. In contrast, our regime must address the singularity of the limit map $\Q_0$. A natural and powerful technique for handling such singularities---well studied in the context of harmonic maps---is blow-up analysis. Specifically, one zooms in at potential singular points and analyzes the properties of the rescaled limits to understand local behavior. 

To prove our main theorem, we refine the standard blow-up framework used in harmonic map analysis. Unlike the harmonic maps, the parameter $ \va\in(0,1) $ complicates the blow-up analysis. Intuitively, if $\Q_\va$ is a local minimizer of the Landau-de Gennes energy \eqref{LdG} on $\om$, then for any $r>0$ and $x\in\om$, the rescaled map $y\mapsto \Q_\va(x+ry)$ is a local minimizer of the same functional with parameter $\f{\va}{r}$. In order to recover the limiting minimizer of the Dirichlet energy \eqref{Dirichletenergy}, we mainly consider sequences satisfying $ \f{\va_i}{r_i}\to 0 $, while other regimes are treated by alternative arguments.

The key step in our proof is to control the neighborhoods of points where the sequence $\{\Q_\va\}$ exhibits irregular behavior, characterized by either
$$
|\na\Q_{\va}|\gg 1,\quad\text{or}\quad\dist(\Q_{\va},\cN)\gtrsim 1.
$$
Inspired by Naber and Valtorta’s work on approximate harmonic maps \cite{NV18}, we observe that within any ball $ B_r(x) $, these “bad points” must in fact lie inside a smaller ball $ B_{\beta r}(y) $, for some $ y\in B_{2r}(x) $ and sufficiently small scale $\beta$, and satisfy a pinching estimate of the form
$$
r^{2-n}\int_{B_r(x)}e_{\va}(\Q_{\va})-\(\f{r}{2}\)^{2-n}\int_{B_{\f{r}{2}}(x)}e_{\va}(\Q_{\va})\ll 1.
$$
Using an iterative covering argument based on this pinching property, we derive sharp estimates on the size of the “bad set”. Consequently, we obtain enhanced regularity of $\Q_\va$, leading to the desired $L^p$ convergence \eqref{Lpconvergence}. The convergence-rate eatimate in \eqref{fQconvergence} follows similarly to the approach in \cite{MSW19,WWW25}.

\subsection{Organization of this paper}
We organize the paper as follows. In Section \ref{pre}, we will present some primary ingredients in the proof, such as the modified monotonicity formula, the partial regularity, and characterization of the bad behavior of minimizers. In Section \ref{pf}, we introduce a key covering lemma and give the proof of our main theorem and the associated corollary. In Section \ref{sharpsection}, we present a concrete to illustrate the sharpness of Theorem \ref{main}.

\subsection{Notations and conventions} We use the following conventions in this paper.
\begin{itemize}
\item Throughout this paper, we denote positive constants by $ C $. To highlight dependence on parameters $ a_1,a_2,... $, we may write $ C(a_1,a_2,...) $, noting that its value may vary from line to line.
\item We will use the Einstein summation convention throughout this paper, summing the repeated index without the sum symbol.
\item Let $ \A,\B\in\MM^{3\times 3} $. Their inner product is $ \A: \B:=\A_{ij}\B_{ij} $.
\item For $ \n,\m\in\R^3 $, we let $ \n\otimes\m\in\MM^{3\times 3} $ with $ (\n\otimes\m)_{ij}=\n_i\m_j $.
\item Assume that $ \A,\B:\om\subset\R^3\to\mathbb{M}^{3\times 3} $ are two differentiable matrix valued functions. The gradient $ \A $ is $ \na\A:=(\pa_1\A,\pa_2\A,\pa_3\A) $. Furthermore, $
\na\A:\na \B:=\pa_k\A_{ij}\pa_k\B_{ij} $. In addition, $ |\na\A|^2=\na\A:\na\A $.
\item In this paper, $ B_r(x):=\{x\in\R^3:|y-x|<r\} $. We will drop $ x $, if it is the original point.
\item $ \I $: identity matrix of order $ 3 $. $ \mathbf{O} $: zero matrix of order $ 3 $.
\item Let $ i,j\in\{1,2,3\}^2 $. $ \delta_{ij}=1 $ if $ i=j $ and $ \delta_{ij}=0 $ if $ i\neq j $.
\item For $ A\subset\R^3 $, the $ r $-neighborhood of $ A $ is
$$
B_r(A):=\bigcup_{y\in A}B_r(y)=\{y\in\R^3:\dist(y,A)<r\}.
$$
\end{itemize}

\section{Preliminary}\label{pre}

\subsection{Euler-Lagrange equation} For $ \va\in(0,1) $, assume that $ \Q_{\va}\in H_{\loc}^1(\om,\Ss_0) $ is a local minimizer of \eqref{LdG}. Let $ \PP\in C_0^{\ift}(\om,\Ss_0) $. It follows from the minimality of $ \Q_{\va} $ that
$$
\left.\f{\ud}{\ud t}\right|_{t=0}E_{\va}(\Q_{\va},\om)=0.
$$
Through simple calculations, by the arbitrariness of $ \PP $, $ \Q_{\va} $ is a weak solution of the Euler-Lagrange equation of \eqref{LdG} 
\be
-\va^2\Delta\Q-a\Q-b\Q^2+\f{b}{3}|\Q|^2\I+c|\Q|^2\Q=\mathbf{O}.\label{EL}
\ee
By Sobolev embedding theorem in $ \R^3 $, we have $ H^1\subset L^6 $, and then $ \Delta\Q_{\va}\in L_{\loc}^2 $. Using standard estimate and Sobolev embedding again, $ \Q_{\va}\in W^{2,2}\subset W^{1,6}\subset C^{0,\f{1}{2}} $. As a result, Schauder's estimate implies that $ \Q_{\va}\in C^{\ift}(\om,\Ss_0) $. Consequently, in our paper, for \eqref{EL}, we only consider smooth solutions.

\subsection{Monotonicity formula} A crucial ingredient in studying minimizers of the Landau-de Gennes energy functional \eqref{LdG} is the monotonicity formula. First introduced in \cite{MZ10}, it implies that if $ \Q_{\va} $ is a smooth solution of \eqref{EL}, then $ r^{-1}E_{\va}(\Q_{\va},B_r(x)) $ is non-decreasing. In our paper, on the other hand, based on the original form of the monotonicity, we apply a modification with a radial weight. The modification will simplify the analysis. For the more intuitions and remarks on similar adjustments, we refer to \cite[Section 1.2 and 2.1]{FWZ24}.

\begin{defn}\label{defnofphi}
Let $ \phi\in C^{\ift}([0,+\ift),\R_{\geq 0}) $ such that the following properties hold.
\begin{enumerate}[label=$(\theenumi)$]
\item $ \supp\phi\subset[0,10) $, $ \phi(t)\geq 1 $ for any $ t\in[0,8] $, and $ \phi(0)=60 $.
\item For any $ t\in[0,+\ift) $, $ \phi(t)\geq 0 $ and $ |\phi'(t)|\leq 100 $.
\item $ -2\leq\phi'(t)\leq -1 $ for any $ t\in[0,8] $.
\item For any $ t\in\R_+ $, $ \phi'(t)\leq 0 $.
\end{enumerate}
\end{defn}

Let $ \Q\in H^1(\om,\Ss_0) $, $ x\in\om $, and $ 0<r<\f{1}{10}\dist(x,\pa\om) $. Define
$$
\Theta_r^{\phi}(\Q,x):=\f{1}{r}\int e_{\va}(\Q)\phi\(\f{|y-x|^2}{r^2}\)\ud y.
$$
The modified monotonicity formula is as follows.

\begin{prop}\label{Monotone}
Assume that $ \Q_{\va}:\om\to\Ss_0 $ is a smooth solution of \eqref{EL}. Let $ x\in\om $ and $ 0<r<R<\f{1}{10}\dist(x,\pa\om) $. Then
\be
\begin{aligned}
&\Theta_R^{\phi}(\Q_{\va},x)-\Theta_r^{\phi}(\Q_{\va},x)\\
&\quad\quad=\int_r^R\[-\f{2}{\rho^2}\int\left|\f{y-x}{\rho}\cdot\na\Q_{\va}\right|^2\phi'\(\f{|y-x|^2}{\rho^2}\)+\f{2}{\va^2\rho^2}\int f(\Q_{\va})\phi\(\f{|y-x|^2}{\rho^2}\)\]\ud\rho.
\end{aligned}\label{Monotone1}
\ee
\end{prop}
\begin{proof}
Up to a translation, let $ x=0 $. For simplicity, define
$$
\phi_{\rho}(y):=\phi\(\f{|y|^2}{\rho^2}\)\quad\text{and}\quad\phi_{\rho}'(y):=\phi'\(\f{|y|^2}{\rho^2}\).
$$
We have the stress-energy identity (see \cite[Lemma 22]{Can17})
$$
\pa_j(e_{\va}(\Q_{\va})\delta_{ij}-\pa_i\Q_{\va}:\pa_j\Q_{\va})=0.
$$
Testing the above equality with $ y\phi_{\rho} $, we have
$$
\int(y_j\phi_{\rho}\pa_je_{\va}(\Q_{\va})-y_i\phi_{\rho}\pa_j(\pa_i\Q_{\va}:\pa_j\Q_{\va}))=0.
$$
Applying integration by parts,
\be
\int\f{2}{\rho^2}\phi_{\rho}'|y\cdot\na\Q_{\va}|^2=\int e_{\va}(\Q_{\va})\phi_{\rho}+\int\f{2}{\va^2}f(\Q_{\va})\phi_{\rho}+\int\f{2|y|^2}{\rho^2}\phi_{\rho}'e_{\va}(\Q_{\va}).\label{addingone}
\ee
On the other hand,
\begin{align*}
\f{\ud}{\ud\rho}\(\f{1}{\rho}\int e_{\va}(\Q_{\va})\phi_{\rho}\)&=-\f{1}{\rho^2}\int e_{\va}(\Q_{\va})\phi_{\rho}-\f{2}{\rho^4}\int e_{\va}(\Q_{\va})|y|^2\phi_{\rho}'\\
&=-\f{2}{\rho^4}\int|y\cdot\na\Q_{\va}|^2\phi_{\rho}'+\f{2}{\va^2\rho^2}\int f(\Q_{\va})\phi_{\rho},
\end{align*}
where for the second inequality, we have used \eqref{addingone}. It implies \eqref{Monotone1}.
\end{proof}

\subsection{Compactness} Recall that for a sequence of minimizers of \eqref{LdG}, up to a subsequence, it converges to a minimizer of \eqref{Dirichletenergy} strongly in $ H_{\loc}^1 $. For later use, we summarize such properties in the following result. The proof follows from the application of a comparison map using interpolation results developed by Luckhaus in \cite{Luc88}. One can also refer to \cite[Proposition 47]{Can17} for a detailed reasoning.

\begin{prop}\label{H1convergence}
Assume that $ \{\Q_{\va}\}_{\va\in(0,1)}\subset H_{\loc}^1(\om,\Ss_0) $ is a sequence of local minimizers of \eqref{LdG} such that for any $ K\subset\subset\om $,
$$
\sup_{\va\in(0,1)}\(E_{\va}(\Q_{\va},K)+\|\Q_{\va}\|_{L^{\ift}(K)}\)<+\ift.
$$
Then, there exists $ \va_i\to 0^+ $ such that
\begin{gather*}
\Q_{\va_i}\to\Q_0\text{ strongly in }H_{\loc}^1(\om,\Ss_0),\\
\f{1}{\va_i^2}f(\Q_{\va_i})\to 0\text{ strongly in }L_{\loc}^1(\om),
\end{gather*}
where $ \Q_0\in H_{\loc}^1(\om,\cN) $ is a local minimizer of \eqref{Dirichletenergy}.
\end{prop}

\subsection{Regularity and partial regularity} In this subsection, we present some regularity results for minimizers of \eqref{LdG} or solutions of the Euler-Lagrange equation \eqref{EL}.

\begin{lem}\label{Apriori}
Let $ \va\in(0,1) $, $ M,r>0 $, and $ x\in\R^3 $. Assume that $ \Q_{\va}:B_{2r}(x)\to\Ss_0 $ is a weak solution of \eqref{EL} with $
\|\Q_{\va}\|_{L^{\ift}(B_{2r}(x))}\leq M $. Then, $ \Q_{\va} $ is smooth and satisfies
\be
\|\na\Q_{\va}\|_{L^{\ift}(B_r(x))}\leq C\(\f{1}{\va}+\f{1}{r}\),\label{Apriori1}
\ee
where $ C>0 $ depends only on $ a,b,c $, and $ M $.
\end{lem}
\begin{proof}
The smoothness follows from the standard Schauder's inequality. Moreover, \cite[Lemma A.1]{BBH93} implies that for a scalar function $ u $, solving $ -\Delta u=g $ in a bounded domain $ U\subset\R^n $, there is
\be
|\na u(y)|^2\leq C\left(\|g\|_{L^{\ift}(U)}\|u\|_{L^{\ift}(\om)}+\f{1}{\dist^2(y,\pa U)}\|u\|_{L^{\ift}(U)}^2\right),\label{A1lemma}
\ee
where $ C=C(n)>0 $. Fix $ i,j\in\{1,2,3\} $. Choose $ U=B_{2r}(x) $, $ y\in B_{2r}(x) $, and define
$$
g_{ij}:=-\f{1}{\va^2}\(a(\Q_{\va})_{ij}+b(\Q_{\va}^2)_{ij}-\f{b}{3}|\Q_{\va}|^2\I_{ij}-c|\Q_{\va}|^2(\Q_{\va})_{ij}\).
$$
By \eqref{EL}, $ \Delta(\Q_{\va})_{ij}=g_{ij} $. Noting that $ \|\Q_{\va}\|_{L^{\ift}}(B_{2r}(x))\leq M $, the estimate \eqref{Apriori1} follows directly from the application of \eqref{A1lemma}.
\end{proof}

Recall that for harmonic maps, the partial regularity theory by Schoen and Uhlenbeck \cite{SU82} establishes that if the localized, scale-invariant energy is sufficiently small, then the minimizer must be smooth around that point. In the context of the Landau-de Gennes model, analogous partial regularity properties holds. We summarize these results as follows.

\begin{lem}\label{partialregularity1}
Let $ \va\in(0,1) $, $ M,r>0 $, and $ x\in\R^3 $. Assume that $ \Q_{\va}:B_{4r}(x)\to\Ss_0 $ is a local minimizer of \eqref{LdG} with $ \|\Q_{\va}\|_{L^{\ift}(B_{4r}(x))}\leq M $. For any $ \delta>0 $, there are $ \eta,\Lda>0 $, depending only on $ a,b,c,\delta $, and $ M $ such that if $ r\in(\Lda\va,1) $ and
$$
r^{-1}E_{\va}(\Q_{\va},B_{2r}(x))<\eta,
$$
then
$$
\|\dist(\Q_{\va},\cN)\|_{L^{\ift}(B_r(x))}<\delta.
$$
\end{lem}
\begin{proof}
Assume that the result is not true. There exist $ r_i\to 0^+ $, $ x_i\in\R^3 $, and a sequence of minimizers of \eqref{LdG}, denoted by $ \{\Q_{\va_i}\} $ such that
\be
r_i^{-1}E_{\va_i}(\Q_{\va_i},B_{2r_i}(x))<i^{-1}\quad\text{and}\quad\wt{\va}_i:=\f{\va_i}{r_i}\to 0^+,\label{constantcondition}
\ee
but there is $ \delta_0>0 $ with
\be
\|\dist(\Q_{\va_i},\cN)\|_{L^{\ift}(B_{r_i}(x_i))}\geq\delta_0.\label{distgeq}
\ee
Define $
\wt{\Q}_{\wt{\va_i}}(y):=\Q_{\va_i}(x+r_iy) $. Given \eqref{constantcondition}, we have,
$$
\wt{\Q}_{\wt{\va}_i}\to\CC_0\text{ strongly in }H^1(B_2,\Ss_0),
$$
where $ \CC_0\in\cN $ is a constant. Similar arguments in \cite[Proposition 4]{MZ10} imply that $ \wt{\Q}_{\wt{\va}_i}\to\CC_0 $ uniformly in $ B_r(x) $, contradicting \eqref{distgeq}.
\end{proof}

\begin{lem}\label{partialregularity2}
Suppose that $ \va,M,r,x $, and $ \Q_{\va} $ are the same as those in Lemma \ref{partialregularity1}. There are $ \eta,\Lda>0 $, depending only on $ a,b,c $, and $ M $ such that if $ r\in(\Lda\va,1) $ and
\be
r^{-1}E_{\va}(\Q_{\va},B_{2r}(x))<\eta,\label{r1EvaQva}
\ee
then
$$
r^2\|e_{\va}(\Q_{\va})\|_{L^{\ift}(B_r(x))}\leq C(\eta)
$$
where $ C(\eta)>0 $ depends only on $ a,b,c,M $, and $ \eta $. Moreover,
$$
\lim_{\eta\to 0^+}C(\eta)=0.
$$
\end{lem}
\begin{proof}
For $ \delta>0 $ to be determined later, we can choose $ (\eta,\Lda)=(\eta,\Lda)(a,b,c,\delta,M)>0 $ as in Lemma \ref{partialregularity1} such that if $ r\in(\Lda\va,1) $ with \eqref{r1EvaQva}, then $ \dist(\Q_{\va},\cN)<\delta $ in $ B_{\f{3r}{2}}(x) $. Applying \cite[Lemma 7]{MZ10}, we can choose appropriate $ \delta=\delta(a,b,c,M)>0 $ and set $ \eta>0 $ smaller if necessary to deduce
$$
r^2\|e_{\va}(\Q_{\va})\|_{L^{\ift}(B_r(x))}\leq C(a,b,c,M).
$$
The dependence of the constants $ C $ and $ \eta $ follows from a compactness argument. 
\end{proof}

\subsection{Characterization of the bad behavior of minimizers}

We first introduce the regular scale to give a quantitative characterization of regularity. 

\begin{defn}[Regular scale]\label{regular}
Let $ \va\in(0,1) $. Assume that $ \Q_{\va}\in H^1(B_2,\Ss_0) $ is a local minimizer of \eqref{LdG}. For $ x\in B_1 $, we define 
$$
r(\Q_{\va},x):=\sup\left\{0\leq r\leq 1:r^2\|e_{\va}(\Q_{\va})\|_{L^{\ift}(B_r(x))}\leq 1\right\}.
$$
\end{defn}

For a minimizer $ \Q_{\va} $ of \eqref{LdG} in $ \om\subset\R^3 $, Lemma \ref{partialregularity1} and \ref{partialregularity2} motivate us that the ``bad points" are those where $ e_{\va}(\Q_{\va}) $ and $ \dist(\Q_{\va},\cN) $ are large. To this reason, for $ \Q_{\va} $ as in \eqref{regular}, we define the collection of \emph{bad points} with parameters $ \delta,r>0 $ as
$$
\Bad(\Q_{\va};r,\delta):=\left\{y\in\om:r(\Q_{\va},y)<r\right\}\cup\left\{y\in\om:\dist(\Q_{\va},\cN)>\delta\right\}.
$$
The following lemma implies that if a minimizer $ \Q_{\va} $ is invariant in one direction in $ B_r(x) $ and the difference of density $ \Theta_{\cdot}^{\phi}(\Q_{\va},x) $ is small at two scales, then it behaves ``well" near $ x $.

\begin{lem}\label{regularscalelem}
Let $ \va\in(0,1) $, $ M,r>0 $, and $ x\in\R^3 $. Assume that $ \Q_{\va}\in H^1(B_{10r}(x),\Ss_0) $ is a local minimizer of \eqref{LdG} with 
$$
r^{-1}E_{\va}(\Q_{\va},B_{10r}(x))+\|\Q_{\va}\|_{L^{\ift}(B_{10r}(x))}\leq M.
$$
For any $ \delta>0 $, there are $ \eta,\Lda>0 $, depending only on $ a,b,c,\delta $, and $ M $ such that if $ r\in(\Lda\va,1) $,
\be
\Theta_r^{\phi}(\Q_{\va},x)-\Theta_{\f{r}{2}}^{\phi}(\Q_{\va},x)<\eta,\label{pincheta1}
\ee
and
$$
\inf_{v\in\Ss^2}\f{1}{r}\int_{B_r(x)}|v\cdot\na\Q_{\va}|^2<\eta,
$$
then $ \dist(\Q_{\va},\cN)<\delta $ in $ B_{\f{r}{2}}(x) $ and $ r(\Q_{\va},x)\geq\f{r}{2} $. In particular $ x\notin\Bad\(\Q_{\va};\f{r}{2},\delta\) $.
\end{lem}
\begin{proof}
Assume that the result is not true. We have a sequence of local minimizers $ \{\Q_{\va_i}\} $, $ \{r_i\}\subset\R_+ $, and $ \{x_i\}\subset\R^3 $ such that $
\wt{\va}_i:=\f{\va_i}{r_i}<i^{-1} $,
\be
\Theta_{r_i}^{\phi}(\Q_{\va_i},x_i)-\Theta_{\f{r_i}{2}}^{\phi}(\Q_{\va_i},x_i)<i^{-1},\label{assum1}
\ee
and
\be
\inf_{v\in\Ss^2}\f{1}{r_i}\int_{B_{r_i}(x_i)}|v\cdot\na\Q_{\va_i}|^2<i^{-1},\label{assum2}
\ee
but $ r(\Q_{\va_i},x_i)<\f{r_i}{2} $ or $ \dist(\Q_{\va_i},\cN)\geq\delta>0 $ in $ B_{\f{r_i}{2}}(x_i) $. Let $ \wt{\Q}_{\wt{\va}_i}(y):=\Q_{\va_i}(x_i+r_iy) $. Up to a subsequence,
$$
\wt{\Q}_{\wt{\va}_i}\to\wt{\Q}_0\in\cN\text{ strongly in }H_{\loc}^1(\R^d,\Ss_0),
$$
where $ \wt{\Q}_0 $ is a local minimizer of \eqref{Dirichletenergy}. Using \eqref{assum2}, we see that $ \wt{\Q}_0 $ is invariant in one direction and it degenerates to a two-dimensional minimizer, which is smooth. \eqref{assum1}, on the other hand, together with \eqref{Monotone} implies that $ \wt{\Q}_0 $ is homogeneous. Combining all above, $ \wt{\Q}_0 $ must be a constant. The contradiction follows from Lemma \ref{partialregularity1} and \ref{partialregularity2}. 
\end{proof}

Furthermore, the condition \eqref{pincheta1} in Lemma \ref{regularscalelem} is somehow unnecessary. We conclude it in the following lemma.

\begin{lem}\label{kplus1}
Let $ \va\in(0,1) $, $ \delta,M,r>0 $, and $ x\in B_2 $. Assume that $ \Q_{\va}\in H^1(B_{40},\Ss_0) $ is a local minimizer of \eqref{LdG}, satisfying 
$$
E_{\va}(\Q_{\va},B_{40})+\|\Q_{\va}\|_{L^{\ift}(B_{40})}\leq M.
$$
There exist $ \eta,\Lda>0 $, depending only on $ a,b,c,\delta $, and $ M $ such that if
\be
\inf_{v\in\Ss^2}\f{1}{r}\int_{B_r(x)}|v\cdot\na u|^2<\eta,\label{detaVsmall}
\ee
then the following property holds. For any $ y\in B_{\f{r}{2}}(x) $ with $ r\in(\Lda\va,1) $, there is $ r_y\in[\eta^{\f{1}{2}}r,1] $ such that $ \dist(\Q_{\va},\cN)<\delta $ in $ B_{\f{r_y}{2}}(y) $ and $ r(\Q_{\va},y)\geq\f{r_y}{2} $.
\end{lem}
\begin{proof}
Given Proposition \ref{Monotone}, we can apply a dyadic decomposition of the radius $ r $ to get
\be
\Theta_{r_y}^{\phi}(\Q_{\va},y)-\Theta_{\f{r_y}{2}}^{\phi}(\Q_{\va},y)<\f{C(a,b,c,M)}{\log(\eta^{-1})},\label{pinchsmall}
\ee
for some $ r_y\in[\eta^{\f{1}{2}}r,\f{r}{2}] $. It follows from \eqref{detaVsmall} that
$$
\inf_{v\in\Ss^2}\f{1}{r_y}\int_{B_{r_y}(y)}|v\cdot\na\Q_{\va}|^2\leq\inf_{v\in\Ss^2}\f{1}{r}\int_{B_{r}(x)}|v\cdot\na\Q_{\va}|^2\leq\eta^{\f{1}{2}}.
$$
This, together with Lemma \ref{regularscalelem} and \eqref{pinchsmall}, implies that for appropriate 
$$
(\eta,\Lda)=(\eta,\Lda)(a,b,c,\delta,M)>0,
$$
the desired properties hold when $r\geq\Lda\va$.
\end{proof}

The result below is the key observation in our paper. 

\begin{prop}\label{Fprop}
Let $ \beta\in(0,\f{1}{2}) $, $ \va\in(0,1) $, $ \delta,M,r>0 $, and $ x\in\R^3 $. Assume that $ \Q_{\va}\in H^1(B_{20r}(x),\Ss_0) $ is a local minimizer of \eqref{LdG}, satisfying 
$$
r^{-1}E_{\va}(\Q_{\va},B_{20r}(x))+\|\Q_{\va}\|_{L^{\ift}(B_{20r}(x))}\leq M.
$$
There exist $ \eta,\eta',\Lda>0 $ depending only on $ a,b,\beta,c,\delta $, and $ M $ such that if there exists $ y\in B_{2r}(x) $ with
\be
\Theta_r^{\phi}(\Q_{\va},y)-\Theta_{\f{r}{2}}^{\phi}(\Q_{\va},y)<\eta\label{pincheta}
\ee
and $ r\in(\Lda\va,1) $, then
$$
\Bad(\Q_{\va};\eta'r,\delta)\cap B_r(x)\subset B_{2\beta r}(y).
$$
\end{prop}
\begin{proof}
If $ z\in B_r(x)\backslash B_{2\beta r}(y) $, we choose $ \sg=\sg(\beta)>0 $ such that 
$$
B_{\sg r}(z)\subset B_{4r}(y)\cap(B_r(x)\backslash B_{2\beta r}(y)).
$$
Given Proposition \ref{Monotone} and \eqref{pincheta}, we have
\be
\int_{B_{4r}(y)}|(\zeta-y)\cdot\na\Q_{\va}|^2\ud\zeta\leq C\eta r^3.\label{leqCr3}
\ee
By the choice of $ z $, we note $ |z-y|\geq 2\beta r $. As a result,
\begin{align*}
\int_{B_{\sg r}(z)}\left|\f{z-y}{|z-y|}\cdot\na\Q_{\va}\right|^2&\leq \f{C}{r^2}\(\int_{B_{\sg r}(z)}|(\zeta-y)\cdot\na\Q_{\va}|^2\ud\zeta+\int_{B_{\sg r}(z)}|(\zeta-z)\cdot\na\Q_{\va}|^2\ud\zeta\)\\
&\stackrel{\eqref{leqCr3}}{\leq} C\eta r+2\sg^2r\(\f{1}{r}\int_{B_{\sg r}(z)}|\na\Q_{\va}|^2\)\\
&\leq C(\beta,M)(\eta+2\sg^2)r,
\end{align*}
where for the last inequality, we have used Proposition \ref{Monotone} again. Choosing sufficiently small 
$$
(\eta,\Lda^{-1},\sg)=(\eta,\Lda^{-1},\sg)(a,b,\beta,c,\delta,M)>0,
$$
we apply Lemma \ref{kplus1} to deduce that if $ r\in(\Lda\va,1) $, then there is $ \eta'=\eta'(a,b,\beta,c,\delta,M)>0 $ such that
$$
r(\Q_{\va},z)\geq\eta'r\quad\text{and}\quad\dist(\Q_{\va}(z),\cN)<\delta,
$$
completing the proof.
\end{proof}

\section{Proof of main theorem}\label{pf}

\subsection{Covering lemmas} In this subsection, we establish some covering lemmas based on results in previous sections. With the help of these properties, we prove Theorem \ref{main}.

\begin{lem}\label{cover1}
Let $ \delta,M>0 $, $ \va\in(0,1) $, $ 0<r<R\leq 1 $, and $ x_0\in B_2 $. Assume that $ \Q_{\va}\in H^1(B_{40},\Ss_0) $ is a local minimizer of \eqref{LdG}, satisfying 
$$
E_{\va}(\Q_{\va},B_{40})+\|\Q_{\va}\|_{L^{\ift}(B_{40})}\leq M.
$$
There exist $ \eta,\Lda>0 $ depending only on $ a,b,c,\delta $, and $ M $ such that if $ r\in(\Lda\va,1) $, the following properties hold. There is $ B_{2r_x}(x)\subset B_{2R}(x_0) $ with $ r_x\geq r $ such that
$$
\Bad(\Q_{\va};\eta r,\delta)\cap B_R(x_0)\subset B_{r_x}(x).
$$
Moreover, either $ r_x=r $ or
$$
\sup_{y\in B_{2r_x}(x)}\Theta_{\f{r_x}{20}}^{\phi}(\Q_{\va},y)\leq\sup_{y\in B_{2R}(x_0)}\Theta_{R}^{\phi}(\Q_{\va},y)-\eta.
$$
\end{lem}
\begin{proof}
Up to a translation, let $ x_0=0$. For $ x\in B_R $ and $ 0<\rho<R $, we define
$$
F_{\eta}(\Q_{\va};x,\rho):=\left\{y\in B_{2\rho}(x):\Theta_{\f{\rho}{20}}^{\phi}(\Q_{\va},y)>E-\eta\right\},
$$
where
$$
E:=\sup_{y\in B_{2R}}\Theta_{R}^{\phi}(\Q_{\va},y).
$$
There is $ \ell\in\Z_+ $ such that $ 2^{-\ell}R\leq r<2^{-\ell+1}R $. Note that if $ F_{\eta}(\Q_{\va};0,R)=\emptyset $, then the ball $ B_R $ is as desired, so we assume that $ F_{\eta}(\Q_{\va};0,R)\neq\emptyset $. For $ i=1 $, by Proposition \ref{Fprop}, for sufficiently small
$$
(\eta,\Lda^{-1})=(\eta,\Lda^{-1})(a,b,c,\delta,M,\rho)>0,
$$
there is $ x_1\in F_{\eta}(\Q_{\va};0,R) $ such that
$$
\Bad(\Q_{\va};\eta r,\delta)\cap B_R\subset B_{\f{R}{10}}(x_1)\cap B_R.
$$
Choose $ r_1=\f{R}{2} $. If $ F_{\eta}(\Q_{\va};x_1,r_1)=\emptyset $, then the ball $ B_{r_1}(x_1) $ is what we need. On the other hand, we proceed to obtain $ B_{r_2}(x_2) $ by using Proposition \ref{Fprop} again such that $ r_2=\f{R}{2^2} $ and 
$$
\Bad(\Q_{\va};\eta r,\delta)\cap B_R\subset B_{\f{1}{2}\cdot\f{R}{10}}(x_2)\cap B_R.
$$
Repeating the procedure, we will either stop at step $ i\in\Z\cap[1,\ell-1] $ or obtain a ball $ B_{r_{\ell-1}}(x_{\ell-1}) $ such that $ r_{\ell-1}=\f{R}{2^{\ell-1}} $ and
$$
\Bad(\Q_{\va};\eta r,\delta)\cap B_R\subset B_{\f{1}{2^{\ell-2}}\cdot\f{R}{10}}(x_{\ell-1})\cap B_R.
$$
For this case, we let $ r_{\ell}=r $ and $ x_{\ell}=x_{\ell-1} $. $ B_{r_{\ell}}(x_{\ell}) $ satisfies all the required properties.
\end{proof}

\begin{rem}\label{remcover}
By further covering arguments, under the assumption of Lemma \ref{cover1}, there is a collection of balls $ \{B_{r_y}(y)\}_{y\in\cD} $ with $ \inf_{y\in\cD}r_y\geq r $,
$$
\Bad(\Q_{\va};\eta r,\delta)\cap B_R(x_0)\subset\bigcup_{y\in r_y}B_{r_y}(y)
$$
such that either $ r_y=r $ or
\be
\sup_{\zeta\in B_{2r_y}(y)}\Theta_{r_y}^{\phi}(\Q_{\va},\zeta)\leq\sup_{z\in B_{2R}(x_0)}\Theta_{R}^{\phi}(\Q_{\va},z)-\eta.\label{energydrop}
\ee
Moreover, $ \#\cD\leq C $, where $C>0$ is an absolute constant. 
\end{rem}

\begin{lem}\label{cover2}
Under the same assumption of Lemma \ref{cover1}, there exist $ \eta,\Lda>0 $ depending only on $ a,b,c,\delta,M $ such that if $ r\in(\Lda\va,1) $, then we have $ \{x_i\}_{i=1}^N\subset B_R(x_0) $, satisfying
$$
\Bad(\Q_{\va};\eta r,\delta)\cap B_R(x_0)\subset\bigcup_{i=1}^NB_{r}(x_i),
$$
where $ N\in\Z_+ $ depends only on $ a,b,c,\delta $, and $ M $. In particular,
\be
\cL^3(B_r(\Bad(\Q_{\va};\eta r,\delta)\cap B_1))\leq Cr^3,\label{Minkowski}
\ee
where $ C>0 $ depends only on $ a,b,c,\delta $, $ M $.
\end{lem}
\begin{proof}
We apply Remark \ref{remcover} to $ B_R(x_0) $. For the balls in the covering with radius larger than $r$, we apply results in Remark \ref{remcover} again. One observes an energy drop of $ \eta $ at each step, as shown in the \eqref{energydrop}, so the procedure concludes after a finite number of steps, as Proposition \ref{Monotone} demonstrates, thus completing the proof.
\end{proof}

\subsection{Proof of Theorem \ref{main}}

The strong convergence in $ H_{\loc}^1(\om,\Ss_0) $ is due to Lemma \ref{H1convergence}. Regarding the $ L^p $ convergence, we show that for any $ K\subset\subset\om $,
\be
\sup_{\va\in(0,1)}\|\na\Q_{\va}\|_{L^{3,\ift}(K)}\leq C(a,b,c,K,M),\label{QvaL3ift}
\ee
where $ L^{3,\ift} $ is the Lorentz space and
$$
\|\na\Q_{\va}\|_{L^{3,\ift}(K)}:=\sup_{t>0}[\cL^3(\{x\in K:|\na\Q_{\va}|>t\})]^{\f{1}{3}}.
$$
The convergence property in $ L^p(\om,\Ss_0) $ follows directly from the interpolation $ L^{3,\ift}\subset L^q $ with $ 1<q<3 $ and Sobolev embedding theorem.

Without loss of generality, let $ K=B_{\f{1}{2}} $, $ \om=B_{40} $, and
$$
E_{\va}(\Q_{\va},B_{40})+\|\Q_{\va}\|_{L^{\ift}(B_{40})}\leq M.
$$
By \eqref{Minkowski}, for $ r\in(\Lda\va,1) $,
\be
\cL^3(B_r((\{y:r(\Q_{\va},y)<\eta r\})\cap B_1))\leq C(a,b,c,M)r^3,\label{L3abc}
\ee
where $ (\eta,\Lda)=(\eta,\Lda)(a,b,c,M)>0 $. Assume that $ 0<r\leq\Lda\va $. Lemma \ref{Apriori} implies that for any $ y\in B_1 $,
$$
r|\na\Q_{\va}(y)|\leq C\(\f{r}{\va}+1\)\leq C(\Lda+1).
$$
As a result, $ r(\Q_{\va},\cdot)>c_0r $ in $ B_{\f{3}{4}} $, where $ c_0=c_0(a,b,c,M)>0 $. Choose $ \eta\in(0,c_0) $, we see that for any $ r\in(0,1) $, \eqref{L3abc} holds. Letting $ r=t^{-1} $,  when $ t>0 $, we have
$$
\cL^3(\{y\in B_1:|\na\Q_{\va}(y)|>t\})\leq C(a,b,c,M)t^{-3},
$$
implying \eqref{QvaL3ift}.

It remains to show the second point. We still assume that $ K=B_{\f{1}{2}} $ and $ \om=B_{40} $. Fix $ 0<\nu<1 $. For any $ \va>0 $ with $ \Lda\va<1 $, there is $ n(\va)\in\Z_+ $ such that $ \nu^{n(\va)-1}\in[\Lda\va,\nu^{-1}\Lda\va] $. Then
\be
\begin{aligned}
\int_{B_1}f(\Q_{\va})\ud x&\leq\int_{B_{\Lda\va}(\Bad(\Q_{\va};\eta\Lda\va,\delta)\cap B_1)} f(\Q_{\va})\ud x \\
&\quad\quad+\sum_{j=0}^{n(\va)-2}\int_{\cA_j} f(\Q_{\va})\ud x+\int_{B_1\backslash B_1(\Bad(\Q_{\va};\eta,\delta))}f(\Q_{\va})\ud x
\end{aligned}\label{use}
\ee
where
$$
\cA_j:=B_{\nu^j}(\Bad(\Q_{\va};\eta\nu^j,\delta)\cap B_1)\backslash B_{\nu^{j+1}}(\Bad(\Q_{\va};\eta\nu^{j+1},\delta)\cap B_1).
$$
By \cite[Corollary 2]{NZ13}, we have
$$
f(\Q_{\va})\leq C(a,b,c,M)\va^4\nu^{-4(j+1)}
$$
in $\cA_j$.\footnote{Indeed, \cite[Corollary 2]{NZ13} is not a scaling invariant form and we need to apply a scaling invariant result. For such a modification, see \cite[Lemma 3.3]{WZ24} for a similar setting.} This, together with Lemma \ref{cover2} and \eqref{use} implies that
$$
\int_{B_1}f_{\va}(\Q_{\va})\ud x\leq C\(\va^3+\sum_{j=0}^{n(\va)-2}\va^{4}\nu^{-(j+1)}+\va^4\)\leq C\va^3,
$$
completing the proof.

\subsection{Proof of Corollary \ref{Maincor}} We first prove the following auxiliary lemma on the behavior of convergence of global minimizers $ \{\Q_{\va}\}_{\va\in(0,1)} $ near the boundary $ \pa\om $.

\begin{lem}\label{Auxlemma}
Under the assumption of Corollary \ref{Maincor}, for any $ \delta>0 $, there exists $ \delta'>0 $ depending only on $ a,b,c,\delta,\om $, and $ \Q_b $ such that for any $ y\in\om $ with $ \dist(y,\pa\om)<\delta' $, if $ \va\in(0,\delta') $, then
$$
|\na\Q_{\va}(y)|\leq C\quad\text{and}\quad\dist(\Q_{\va}(y),\cN)<\delta,
$$
where $ C>0 $ depends only on $ a,b,c,\om $, and $ \Q_b $.
\end{lem}
\begin{proof}
For $ \sg>0 $, define
$$
\om_{\sg}:=\{y\in\om:\dist(y,\pa\om)<\sg\}.
$$
Assume that the result is false. Then, there is a sequence of minimizers $ \Q_{\va_i} $ of \eqref{globalminimizer}, $ \delta_i'\to 0^+ $, $ y_i\in\om $ with $ \dist(y_i,\pa\om)<\delta_i' $, $ \va_i\in(0,\delta_i') $ such that 
$$
\lim_{i\to+\ift}|\na\Q_{\va_i}(y_i)|=+\ift\quad\text{or}\quad\inf_{i\in\Z_+}\dist(\Q_{\va_i}(y_i),\cN)\geq\delta>0.
$$
Up to a subsequence, \cite[Lemma 3]{MZ10} implies that $ \Q_{\va_i}\to\Q_0 $ strongly in $ H^1(\om,\Ss_0) $, where $ \Q_0 $ minimizes \eqref{globalminimizer2}. By \cite[Theorem 2.7]{SU83} (also see \cite[Theorem 2.4.1]{LW08}), there exists $ \delta_0>0 $, depending only on $ a,b,c,\om,\Q_0 $, and $ \Q_b $ such that $ u\in\om_{\delta_0} $, where $ \om_{\delta_0} $. Since $ \sing(\Q_0)\cap\om_{\delta_0}=\emptyset $, it follows from \cite[Proposition 6]{MZ10} and \cite[Proposition 3]{NZ13} that 
$$
\dist(\Q_{\va}(y_i),\cN)\to 0^+\quad\text{ and }\quad \sup_{i\in\Z_+}\|\na\Q_{\va_i}\|_{L^{\ift}(\om_{\delta_0})}<+\ift,
$$
contradicting to the original assumptions.
\end{proof}

\begin{proof}[Proof of Corollary \ref{Maincor}]
The convergence $ \Q_{\va_i}\to\Q_0 $ follows from \cite[Lemma 3]{MZ10}. To prove \eqref{cor1}, we first fix $ p\in(1,+\ift) $. By \cite[Proposition 3]{MZ10}, we have $ \|\Q_{\va}\|\leq C(a,b,c) $ for any $ \va\in(0,1) $. Given $ \va'>0 $, choose $ \delta=\delta(a,b,c,\va',\om)>0 $ such that
\be
\|\Q_{\va}-\Q_0\|_{L^p(\om_{\delta})}<\f{\va'}{2}.\label{QvaQ0small1}
\ee
Applying \eqref{Lpconvergence}, if $ i\in\Z_+ $ is sufficiently large, then
$$
\|\Q_{\va_i}-\Q_0\|_{L^p(\om\backslash\om_{\delta})}<\f{\va'}{2}.
$$
This, together with \eqref{QvaQ0small1}, implies $ \|\Q_{\va}-\Q_0\|_{L^p(\om)}<\va' $ when $ i\in\Z_+ $ is large enough.

For \eqref{cor1}, let $ \sg\in(0,\f{1}{100}\diam(\om)) $ be determined later. By Lemma \ref{Auxlemma}, we find $ \delta'=\delta'(a,b,c,\om,\Q_b,\sg)>0 $ such that for any $ \va\in(0,\delta') $,
$$
\|\na\Q_{\va}\|_{L^{\ift}(\om_{\delta'})}\leq C(a,b,c,\om,\Q_b)\quad\text{and}\quad\|\dist(\Q_{\va},\cN)\|_{L^{\ift}(\om_{\delta'})}\leq\sg.
$$
Cover $ \om_{\sg} $ with balls $ \{B_{2\sg}(x_i)\}_{i=1}^N $ such that $ \{x_i\}\subset\pa\om $, where $ N\leq C(a,b,c,\om,\Q_b,\sg) $. For $ y\in B_{\sg}(x_i) $, it follows from \cite[Corollary 2]{NZ13} that 
\be
0\leq f(\Q_{\va}(y))\leq C(a,b,c,\om,\Q_b,\sg)\va^4,\label{fQva4}
\ee
whenever $ \sg=\sg(a,b,c,\om,\Q_b)>0 $ is sufficiently small. The inequality \eqref{fQva4} and \eqref{fQconvergence} directly imply \eqref{cor2}.
\end{proof}

\section{Sharpness of Theorem \ref{main}}\label{sharpsection}

Assume that $ \Q_{\va}\in H^1(B_1,\Ss_0) $ is a minimizer of \eqref{LdG} with $ \Q=\Q_{\op{b}} $ on $ \pa B_1 $, where
$$
\Q_{\op{b}}:=s_*\(x\otimes x-\f{1}{3}\I\).
$$
By Corollary \ref{Maincor}, there exists $ \va_i\to 0^+ $ such that
$$
\Q_{\va_i}\to\Q_0=s_*\(\n_0\otimes\n_0-\f{1}{3}\I\),
$$
where $ \n_0 $ is a minimizer of 
$$
\inf_{\n=x\text{ on }\pa B_1}\int_{B_1}|\na\n|^2\ud x.
$$
It follows from \cite[Theorem 7.1]{BCL86} that $ \n_0(x)=\f{x}{|x|} $, where $ \Q_0 $ is the so-called hedgehog solution. Note that $ \f{x}{|x|} $ is not smooth, showing that the convergence in \eqref{Lpconvergence} is sharp. Indeed, if $ \Q_{\va}\to\Q_0 $ uniformly in $ B_{\f{1}{2}} $, then $ \Q_0 $ is continuous at $ 0 $, a contradiction. Since $ \pa B_1 $ is smooth and $ \f{x}{|x|} $ has only one singularity $ 0 $, \cite[Proposition 6]{MZ10} implies that $ \Q_{\va_i}\to\Q_0 $ uniformly in $ B_1\backslash B_{\f{1}{2}} $. Regarding the convergence behavior of $ \Q_{\va_i} $, the following property holds.

\begin{prop}\label{propsharp}
For $ \Q_{\va_i} $ given as above, there is $ C>0 $, depending only on $ a,b $, and $ c $ such that for sufficiently large $ i\in\Z_+ $,
$$
\int_{B_{\f{3}{4}}}f(\Q_{\va_i})\ud x\geq \f{\va_i^3}{C}.
$$
\end{prop}

Using the above result, the estimate \eqref{fQconvergence} is sharp. To show Proposition \ref{propsharp}, we need the following lemmas.

\begin{lem}\label{sharplem1}
There exists $ \eta>0 $, depending only on $ a,b $, and $ c $ such that
$$
T_{\eta}:=\{\Q\in\Ss_0:f(\Q)<\eta\}\subset\{\Q\in\Ss_0:\lda_1(\Q)>\lda_2(\Q)\},
$$
where $ \lda_1(\Q)\geq\lda_2(\Q)\geq\lda_3(\Q) $ are three eigenvalues in order. 
\end{lem}
\begin{proof}
Assume that $ \lda_1(\Q)=\lda_2(\Q)=\lda $. We have $ \lda_3(\Q)=-2\lda $. As a result,
\be
f(\Q)=k-3a\lda^2+2b\lda^3+9c\lda^4:=g(\lda).\label{glda}
\ee
By \cite[Proposition 15]{MZ10},
$$
k=\f{s_*^2}{27}(9a+2bs_*-3cs_*^2).
$$
It follows from simple calculation that $ g(\lda) $ in \eqref{glda} achieves the minimum at 
$$
\lda_*=\f{-b+\sqrt{b^2+24ac}}{12c}.
$$
Moreover, $ g(\lda_*)>0 $. Choosing $ \eta\in(0,g(\lda_*)) $, we obtain the desired property.
\end{proof}

\begin{lem}\label{sharplem2}
There is no continuous map $ \Q:\ol{B}_1\to\cN $ such that 
$$
\Q(x)=s_*\(x\otimes x-\f{1}{3}\I\)\quad\text{on }\pa B_1.
$$
\end{lem}
\begin{proof}
Let $ p:\Ss^2\to\cN $ be the covering map. Indeed, we have 
$$
p(\n)=s_*\(\n\otimes\n-\f{1}{3}\I\)\quad\text{for }\n\in\Ss^2.
$$
Since $ \ol{B}_1 $ is simply connected, for $ \Q\in C(\ol{B}_1,\cN) $, there exists $ \wt{\Q}\in C(\ol{B}_1,\Ss^2) $ such that $ p\circ\wt{\Q}=\Q $ with $ \wt{\Q}(x)=x $ on $ \pa B_1 $, leading to a contradiction since there is no retraction from $\ol{B}_1$ to $\Ss^2$.
\end{proof}

\begin{proof}[Proof of Proposition \ref{propsharp}]
Fix $ \Q_{\va_i} $. We claim that there is $ y_i\in B_1 $ such that $ f(\Q_{\va_i}(y_i))>\eta $, where $ \eta>0 $ is from Lemma \ref{sharplem1}. It is because if for any $ y\in B_1 $, $ f(\Q_{\va_i})(y)<\eta $, Lemma \ref{sharplem1} implies that for any $ y\in\ol{B}_1 $, $ \lda_1(\Q_{\va_i})>\lda_2(\Q_{\va_i}) $. By \cite[Lemma 12]{Can17}, there is a $ C^1 $ nearest point projection 
$$
\Pi:\{\Q\in\Ss_0:\lda_1(\Q)>\lda_2(\Q)\}\to\cN.
$$
Then, $ \Pi\circ\Q_{\va_i}\in C(\ol{B}_1,\cN) $ and 
$$
\Pi\circ\Q_{\va_i}(x)=s_*\(x\otimes x-\f{1}{3}\I\)\quad\text{for any }x\in\pa B_1,
$$
a contradiction to Lemma \ref{sharplem2}. Recalling that $ \Q_{\va_i}\to\Q_0 $ uniformly in $ B_1\backslash B_{\f{1}{2}} $, for $ i\in\Z_+ $ sufficiently large, we have $ y_i\in B_{\f{1}{2}} $. By Lemma \ref{Apriori}, 
$$
|\na\Q_{\va_i}|\leq\f{C(a,b,c)}{\va_i}\quad\text{in }B_{\f{3}{4}}.
$$
There is $ \delta=\delta(a,b,c)\in(0,\f{1}{4}) $ such that
$$
f(\Q_{\va_i}(y))>\f{\eta}{2}\quad\text{for any }y\in B_{\delta\va_i}(y_i).
$$
When $ \va_i\in(0,\f{1}{2}) $, we have $ B_{\delta\va_i}(y_i)\subset B_{\f{3}{4}} $ and then
$$
\int_{B_{\f{3}{4}}}f(\Q_{\va_i})\geq\int_{B_{\delta\va_i}(y_i)}f(\Q_{\va_i})\geq\f{\w_3\eta\delta\va_i^3}{2},
$$
where $\w_3$ denotes the volume of unit balls in $\mathbb R^3$, then completing the proof.
\end{proof}

\section*{Acknowledgment}

The authors would like to express their gratitude to Professor Zhifei Zhang and Professor Zhiyuan Geng for their valuable advice and insightful comments. Additionally, the authors want to thank Professor Arghir Zarnescu for his constructive suggestions on incorporating results such as Corollary \ref{Maincor} into the discussion of global minimizers The authors are partially supported by the National Key R$\&$D Program of China under Grant 2023YFA1008801 and NSF of China under Grant 12288101.

\end{document}